\documentclass[leqno,a4paper]{amsart}

\usepackage{amsmath,amsfonts,amssymb}
\usepackage{verbatim}
\usepackage{url}
\usepackage{tikz}
\usepackage{setspace}
\usepackage[margin=2.7cm]{geometry}
\usepackage[utf8]{inputenc}
\usepackage[T1]{fontenc}
\usepackage[colorlinks]{hyperref}
\definecolor{linkblue}{RGB}{1,1,190}
\definecolor{citered}{RGB}{190,1,1}
\hypersetup{
	linkcolor=linkblue,
	urlcolor=linkblue,
	citecolor=citered
}
\usetikzlibrary{arrows}

\def\d{{\sf d}}

\def\E{{\sf E}}
\def\v{{\sf v}}
\def\HH{{\sf H}}
\def\F{{\mathcal F}}
\def\bd{{\boldsymbol{\cdot}}}
\def\supp{\text{supp}}
\def\stab{\text{stab}}

\def\Z{\mathbb Z}

\makeatletter
\newcommand{\dotprod}{\mathbin{\mathpalette\dotprod@\relax}}
\newcommand{\dotprod@}[2]{%
	\ooalign{$\m@th#1{\LARGE\bd}$\cr\hidewidth$\m@th#1\prod$\hidewidth\cr}%
}

\makeatother

\theoremstyle{plain}
\newtheorem{theorem}{Theorem}[section]
\newtheorem*{theorem*}{Theorem}
\newtheorem{lemma}[theorem]{Lemma}

\newtheorem{proposition}[theorem]{Proposition}
\newtheorem{remark}[theorem]{Remark}

\theoremstyle{definition}

\numberwithin{equation}{section}
\newtheorem{claim}{}[theorem]

\title{On zero-sum problems over metacyclic groups $C_n \rtimes_s C_2$}

\author{Jun Seok Oh, Sávio Ribas, Kevin Zhao, and Qinghai Zhong}
\address{
	Department of Mathematics Education\\
	Jeju National University\\
	Jeju\\
	63243\\
	Republic of Korea
}
\email{junseok.oh@jejunu.ac.kr}

\address{
	Departamento de Matem\'{a}tica\\
	Universidade Federal de Ouro Preto\\
	Ouro Preto, MG\\
	35402-136\\
	Brazil
}
\email{savio.ribas@ufop.edu.br}

\address{
	School of Mathematics and Statistics \&
	China Center for Applied Mathematics of Guangxi\\
	Nanning Normal University\\ Nanning\\
	530100\\
	China
}
\email{zhkw-hebei@163.com}

\address{Institute of Mathematics and Interdisciplinary Sciences, Xidian University, Xi'an 710126, China \&
	Department of Mathematics and Scientific Computing\\
	University of Graz\\
	NAWI Graz\\
	Graz\\
	8010\\
	Austria
}
\email{qinghai.zhong@uni-graz.at, zhong@xidian.edu.cn}

\date{}

\usepackage{xcolor}

\begin{document}

	\maketitle
	
	\begin{abstract}
    Let $G$ be a finite group. A finite collection of elements from $G$, where the order is disregarded and repetitions are allowed, is said to be a product-one sequence if its elements can be ordered such that their product in $G$ equals the identity element of $G$.
    Then, the Gao's constant $\E(G)$ of $G$ is the smallest integer $\ell$ such that every sequence of length at least $\ell$ has a product-one subsequence of length $|G|$. For a positive integer $n$, we denote by $C_n$ a cyclic group of order $n$.
    Let $G = C_n \rtimes_s C_2$ with $s^2\equiv 1\pmod n$ be a metacyclic group. The direct and inverse problems of $\mathsf E(G)$ were settled recently, except for the case that $G=C_{3n_2}\rtimes_s C_2$ with $n_2\neq 1$, $\gcd(n_2,6)=1$, $s\equiv -1 \pmod 3$,  and $s\equiv 1\pmod {n_2}$. In this paper, we complete the remaining case and hence for all metacyclic groups of the form $G=C_n \rtimes C_2$, the Gao's constant and the associated inverse problem are now fully settled (see Theorem \ref{thm:except}).
	\end{abstract}

	\bigskip
	\section{Introduction}

	Let $G$ be a finite group. By a sequence $S$ over $G$, we mean a finite collection of elements from $G$, where the order is disregarded and repetitions are allowed. The {\em zero-sum problems} study the conditions that ensure  a given sequence over $G$ contains a subsequence whose product (or sum, in the abelian case) is the identity of $G$. Such subsequences are called {\em product-one} (or {\em zero-sum}). Originating in classical works of Erd\H os, Ginzburg and Ziv \cite{EGZ}, van Emde Boas and Kruyswijk \cite{vEBK}, and Olson \cite{Ols1,Ols2} during the 1960s, these problems connect to several areas (see, for instance, \cite{AGP,PlSc,GeHK}), and have been extended from abelian to non-abelian groups in the 1980s. This explains why it is common to use multiplicative notation and the term {\em product-one} instead of additive notation and the term {\em zero-sum}.

	The Gao's constant $\E(G)$ is the smallest positive integer $\ell$ such that every sequence of length at least $\ell$ has a product-one subsequence of length $|G|$.
	This invariant has been extensively studied for finite abelian groups, and it is closely related to $\d(G)$, the small Davenport constant of $G$, which is the maximal length of a sequence over $G$ that do not have product-one subsequences (see \cite[Chapter 16]{Gry}). In fact, it is known that $\E(G) \ge \d(G)+|G|$ and equality holds for abelian groups \cite{Ga0} and several others (see \cite{AMR1,AMR,Bas,MLMR,HZ,QL2,GoSa} for non-abelian groups). Zhuang and Gao \cite{ZhGa} conjectured that this equality holds for every finite group.
	
	The {\em direct problem} associated to $\E(G)$ is to determine the exact value of $\E(G)$, while the {\em inverse problem} seeks to describe the structure of sequences with maximal length that do not have product-one subsequences of length $|G|$ (see also \cite{MR1,MR2,MR3,GLQ,OZ1,OZ2,Zh,QL1} for other recent developments on the direct and inverse problems for non-abelian groups).

	Let $C_n$ be a cyclic group of order $n$, and let $G = C_{n} \rtimes_s C_2 = \langle x, y \colon x^{2} = y^{n} = 1_G, yx = xy^{s} \rangle$ with $s^2 \equiv 1 \pmod{n}$,  be a metacyclic group of order $2n$. If $s = 1$ or $s = -1$, then $G \cong C_n \times C_2$ or $G \cong D_{2n}$, a dihedral group of order $2n$.
	If $s \not\equiv \pm 1 \pmod{n}$, then $n = n_1 n_2$ for some $n_1, n_2 \in \mathbb Z$ with $\gcd (n_1, n_2) \in \{ 1, 2 \}$ for which $s \equiv -1 \pmod{n_1}$ and $s \equiv 1 \pmod{n_2}$ (see Remark~\ref{rmk}). With these notation, the direct and inverse problems for $\mathsf E (G)$ have been settled for almost all metacyclic groups of the given form except for a ``small'' (yet infinite) family of cases for which $n_1=3$ and $n_2$ is odd. Note that $D_{2n_1}$ is isomorphic to some quotient group of $G$. By using the inductive method, the inverse structural theorem for $D_{2n_1}$ (see \cite[Theorem 1.2]{OZ1}) is needed.
	 Since the structure of extremal sequences for the group $D_6$ has more possibilities than other groups $D_{2n_1}$ with $n_1>3$, the method used in \cite{AMR} does not work. More precisely, the difficulties are as follows: \cite[Lemma 2.1]{AMR} (see also \cite{Ga}) provides an inverse problem for $\E(C_n)$ which ensures the existence of two elements with high multiplicity. These elements are essential in the argument to partially undo subsequences whose products lie in $\langle y^{n_2} \rangle$, and to construct new subsequences whose products belong to $x \langle y^{n_2} \rangle$ (see Claim A(iii,iv) in the proof of \cite[Proposition 5.1]{AMR}). Moreover, when $n_1 = 3$, the last type of sequence described in Lemma \ref{lem:D6}.2 may occur, and in this case Claim A(iii) in the proof of \cite[Proposition 5.1]{AMR} fails, thereby breaking a crucial step of the method.

	In this paper, we consider this exceptional case with the help of the DeVos-Goddyn-Mohar Theorem (see Lemma \ref{le1}). This theorem allows us to control how concentrated the projection of a sequence $S$ onto $\langle y^3 \rangle$ can be within cosets of a subgroup $H$, defined as the stabilizer of an appropriate set of restricted subproducts arising from this projection. In particular, either this set of subproducts coincides with $\langle y^3 \rangle$, or there exists a unique coset $y^{3t}H$ that contains many terms of $S$. This is a key step in the proof of Proposition \ref{pro1}, which in turn leads to the resolution of the exceptional case stated below.

		\smallskip
		\begin{theorem} \label{thm:main}~
		Let $G = C_{3n_2} \rtimes_s C_2$ be a metacyclic group of order $6n_2$, where $n_2\neq 1$, $\gcd(6, n_2) = 1$,  $s \equiv -1 \pmod{3}$, and $s \equiv 1 \pmod{n_2}$.
			\begin{enumerate}
			\item $\E(G) = 9n_2$.
		
			\smallskip
			\item Let $S \in \F(G)$ be a sequence of length $9n_2-1$. Then, $S$ has no product-one subsequence of length $6n_2$ if and only if there exist $\alpha, \tau \in G$ such that $G = \langle \alpha, \tau \colon \tau^2 = \alpha^n = 1_G, \alpha \tau = \tau \alpha^s \rangle$,  and $S = (\alpha^{t_1})^{[6n_2-1]} \bd (\alpha^{t_2})^{[3n_2-1]} \bd (\tau \alpha^{t_3})$, where $t_1, t_2, t_3 \in \Z$ with $\gcd(t_1-t_2,3n_2) = 1$.
			\end{enumerate}
		\end{theorem}
		
		\smallskip
		 In conjunction with \cite[Theorems 1.6 and 1.7]{AMR}, the Gao's constant and the associated inverse problem are completely settled for all metacyclic groups of the form $C_n \rtimes C_2$ as follows.

		\smallskip
		\begin{theorem} \label{thm:except}~
		Let  $G = C_n \rtimes_s C_2$ be a  metacyclic group of order $2n$, where $n\ge 3$ and  $s^{2} \equiv 1 \pmod{n}$ with $s\not\equiv 1 \pmod n$.
			\begin{enumerate}
			\item $\E(G) = 3n$.
		
			\smallskip
			\item If $S \in \F(G)$ with $|S| = 3n-1$, then $S$ has no product-one subsequence of length $2n$ if and only if there exist $\alpha, \tau \in G$ for which $G = \langle \alpha, \tau \colon \tau^2 = \alpha^n = 1_G, \alpha \tau = \tau \alpha^s \rangle$ and $S$ has one of the following forms;
				\begin{itemize}
				\item for $n = 3$, either $S = 1^{[5]}_G \bd \tau \bd \tau \alpha \bd \tau \alpha^2$, or $S = (\alpha^{k_1})^{[5]} \bd (\alpha^{k_2})^{[2]} \bd \tau \alpha^{k_3}$, where $k_1, k_2, k_3 \in \mathbb Z$ with $\gcd (k_1 - k_2, n) = 1$.
			
				\smallskip
				\item for $n \ge 4$, $S = (\alpha^{t_1})^{[2n-1]} \bd (\alpha^{t_2})^{[n-1]} \bd \tau \alpha^{t_3}$, where $t_1, t_2, t_3 \in \Z$ with $\gcd(t_1-t_2,n) = 1$.
				\end{itemize}
			\end{enumerate}
		\end{theorem}

		\smallskip
		The paper is organized as follows. In Section \ref{sec:defnotlemmas}, we introduce our notation and the preliminary results. In Section~\ref{sec:additive}, we present the results from Additive Theory needed for the proof of our main theorem.
        In Section \ref{sec:proofmainresult}, we prove Theorem \ref{thm:main} after establishing a result on sequences over $G$ of length $9n_2-1$ containing at least two terms from $x \langle y \rangle$.

	\bigskip
	\section{Preliminaries}\label{sec:defnotlemmas}

	For integers $a, b \in \mathbb Z$, $[a, b] = \{ x \in \mathbb Z \colon a \le x \le b \}$ is the discrete interval.
	Let $G$ be a finite group with identity element $1_G$.
	For a subset $A \subset G$ and an element $g \in G$, we let $gA := \{ ga \in G \colon a \in A \}$, and  denote by $\langle A \rangle$ the subgroup of $G$ generated by $A$, and by $\mathsf H (A) := \stab (A) = \{ g \in G \colon gA = A \}$ the {\it stabilizer} of $A$. Then, $\mathsf H (A)$ is a subgroup of $G$, and it is easy to see that $\mathsf H (A) = G$ if and only if $A = G$.
	For a positive  integer $n$, we denote by $C_n$ a cyclic group of order $n$, and by $D_{2n}$ a dihedral group of order $2n$.

We  consider sequences over $G$ as elements of the free abelian monoid $\mathcal F (G)$ with basis $G$, under concatenation as the operation (denoted by $\bd$). In order to avoid confusion between multiplication in $G$ and multiplication in $\mathcal F(G)$, we denote multiplication in $G$ without an explicit symbol and  use brackets for all exponentiation in $\mathcal F (G)$.
Thus, a sequence $S \in \F(G)$ has the form
	\[
	  S \, = \, g_1 \bd \ldots \bd g_k \, = \, \small{\prod}^{\bullet}_{i \in [1,k]} g_i \, = \, \small{\prod}^{\bullet}_{g \in G} g^{[\v_g(S)]} \,,
	\]
	where $g_1, \ldots, g_k \in G$, $\v_g(S)$ is the {\em multiplicity} of $g$ in $S$, and $k = |S| = \sum_{g \in G} \v_g(S)$ is the {\em length} of $S$.
    We denote by $\supp (S) = \{ g \in G \colon \mathsf v_g (S) \ge 1 \}$ the {\it support} of $S$.
	A {\em subsequence} of $S$ is a divisor $T$ of $S$ in $\F(G)$, which is equivalent to $\v_g(T) \le \v_g(S)$ for every $g \in G$. In this case, we write $T \mid S$, and $S \bd T^{[-1]} = \prod^{\bullet}_{g \in G} g^{[\v_g(S) - \v_g(T)]}$ denote the subsequence of $S$ obtained by deleting the elements of $T$ from $S$.
	For a subset $H \subset G$, we denote by $S_H$ the subsequence of $S$ consisting of all elements from $H$.

	For a sequence $S = g_1 \bd \ldots \bd g_k$, the {\em set of products} and the {\em set of $n$-subproducts} of $S$ are denoted by
	\[
	  \pi(S) \, = \, \{  g_{f(1)} \ldots g_{f(k)}  \in G \colon f \text{ is a permutation of }[1,k] \} \quad \text{ and } \quad \Pi_n(S) = \bigcup_{T \mid S \atop |T|=n} \pi(T) \,,
	 \]
	respectively. A generic element of $\pi(S)$ will be denoted by $\sigma(S)$. In particular, if $G$ is abelian, then $\pi (S) = \{ \sigma (S) \}$ is a singleton set. The sequence $S \in \F(G)$ is called
	\begin{itemize}
	\item {\em product-one} if  $1_G \in \pi(S)$, and
	
	\smallskip
	\item {\em product-one free} if $1_G \not\in \cup_{1 \le n \le |S|} \Pi_n(S)$.
	\end{itemize}
	In particular, a product-one sequence of length $k$ (resp., a product-one free sequence of length $k$) is referred to as an {\it $k$-product-one sequence} (resp., {\it $k$-product-one free sequence}).

	If $S = g_1 \bd \ldots \bd g_k$ is a product-one sequence with $1_G = g_1 \ldots g_k$, then $1_G = g_i \ldots g_k g_1 \ldots g_{i-1}$ for all $i \in [1,k]$.
	For a normal subgroup $H$ of $G$, let $\phi_H \colon G \to G/H$ denote the natural homomorphism. By abuse of notation, we also write $\phi_H \colon \F(G) \to \F(G/H)$ for its extension to sequences, that is, $\phi_H ( g_1 \bd \ldots \bd g_k ) = \phi_H (g_1) \bd \ldots \bd \phi_H (g_k)$.
    In this case, we say that $S$ is an {\em $n$-product-$H$ sequence} if $\phi_H (S)$ is a product-one sequence (equivalently, $\pi(S) \cap H \neq \emptyset$) of length $n$, and $S$ is an {\em $n$-product-$H$ free sequence} if $\phi_H(S)$ has no a product-one subsequence of length $n$.
	
	With this notation, we have
	\[
	  \E(G) \, = \, \min\{\ell>0 \colon \text{every $S \in \F(G)$ with $|S| \ge \ell$ has a $|G|$-product-one subsequence}\} \,.
	\]

	Let $G = C_n \rtimes_s C_2$ be a metacyclic group of order $2n$ with $s^{2} \equiv 1 \pmod{n}$.
	Then $n$ can be factored into a product of two integers as follows.
	
	\smallskip
	\begin{remark} \label{rmk}~
	Let $n \ge 3$ be an integer and $s \in \Z$ be such that $s^2 \equiv 1 \pmod n$. An argument similar to that in \cite[Lemma 1.4]{AMR} shows  $n$ admits a factorization of the form $n = n_1n_2$ for which $s \equiv -1 \pmod {n_1}$, $s \equiv 1 \pmod {n_2}$ and $\gcd(n_1,n_2) \in \{1,2\}$. If $n_1,n_2$ are coprime (this occurs, for instance, when $n$ is odd), then the group $G = C_n \rtimes_s C_2 = \langle x,y \colon x^2 = y^n = 1_G, yx = xy^s \rangle$ is isomorphic to $C_{n_2} \times D_{2n_1} = \langle y^{n_1} \rangle \times \langle x, y^{n_2} \rangle$.
	\end{remark}

	\smallskip
	In \cite{AMR}, the authors proved the direct and inverse problems for $\mathsf E (G)$ except for a ``small'' (yet infinite) family of cases; indeed, the only potential exception occurs when $n_1 = 3$ and $n_2$ is odd (see \cite[Theorems 1.6 and 1.7]{AMR}), since the arguments used for larger values of $n_1$ do not work in this case.

	We present here the results related to the Gao's constant, which will be referred to frequently throughout this paper. The following lemma can be found in \cite{EGZ}, \cite[Lemma 2.2]{AMR}, and \cite[Lemma 2.5]{OZ1}.

	\smallskip
	\begin{lemma} \label{inverseCn}~
	Let $n \ge 2$ be an integer and $S \in \F(C_n)$ be a sequence.
		\begin{enumerate}
		\item $\E(C_n) = 2n-1$.
		
		\smallskip
		\item If $|S| = 3n-1$, then $S$ has two disjoint $n$-product-one subsequences. In particular, $S$ has a $2n$-product-one subsequence.
		
		\smallskip
		\item If $|S| = 3n-2$, then $S$ is $2n$-product-one free if and only if there exist a generator $g$ of $C_n$ and $t_1, t_2  \in \Z$ with $\gcd(t_1 - t_2,n) = 1$ such that $S = (g^{t_1})^{[2n-1]} \bd (g^{t_2})^{[n-1]}$.
		\end{enumerate}
	\end{lemma}

	\smallskip
	Although the direct and inverse problems concerning the Gao's constant for all dihedral groups have been completely solved, we present here only the case of the dihedral group $D_6$ of order 6, which we need for our main theorem, and it can be found in \cite[Theorem 8]{Bas} and \cite[Theorem 1.2]{OZ1}.

	\smallskip
	\begin{lemma} \label{lem:D6}~
	Let $S \in \mathcal F (D_6)$ be a sequence.
		\begin{enumerate}
		\item $\E(D_6) = 9$.
		
		\smallskip
		\item If $|S| = 8$, then $S$ is $6$-product-one free if and only if there exist $\alpha, \tau \in D_6$ and $t_1, t_2, t_3 \in \Z$ such that $D_6 = \langle \alpha, \tau \colon \tau^2 = \alpha^3 = 1_G, \alpha \tau = \tau \alpha^{-1} \rangle$, $\gcd(t_1-t_2,3) = 1$ and either $S = (\alpha^{t_1})^{[5]} \bd (\alpha^{t_2})^{[2]} \bd (\tau \alpha^{t_3})$ or $S = 1_G^{[5]} \bd \tau \bd \tau \alpha \bd \tau \alpha^2$.
		\end{enumerate}
	\end{lemma}

	\smallskip
	The following lemma is essentially contained in Claim A in the proof of \cite[Proposition 5.1]{AMR}, but we include its proof here for completeness.

	\smallskip
	\begin{lemma} \label{lem:claima}~
		Let $G = C_n \rtimes_s C_2=\langle x,y \colon x^2 = y^n = 1_G, yx = xy^s \rangle$, where $n\ge 3$ is odd such that $s^2\equiv 1\pmod n$, $S  \in \F(G)$ with $|S| = n_2$ be such that $\pi(S)$ is a singleton set, and $n_1, n_2$ be defined as in Remark $\ref{rmk}$.
		\begin{enumerate}
			\item If $\pi(S) \subset \langle y^{n_2} \rangle$ and $|S_{x\langle y \rangle}| \ge 1$, then $\pi(S) = \{1_G\}$.
			
			\item If $\pi(S) \subset x\langle y^{n_2} \rangle$, then $S = xy^{\beta_1} \bd \ldots \bd xy^{\beta_{2k-1}} \bd y^{\beta_{2k}} \bd \ldots \bd y^{\beta_{n_2}}$ for some $k\ge 1$, where $\beta_1 \equiv \cdots \equiv \beta_{2k-1} \pmod {n_1}$ and $\beta_{2k} \equiv \cdots \equiv \beta_{n_2} \equiv 0 \pmod {n_1}$. Moreover, $\pi(S) = \{xy^{\beta_1 + tn_1}\}$ for some $t \in \Z$ with $\beta_1 + tn_1 \equiv 0 \pmod {n_2}$.
		\end{enumerate}
	\end{lemma}
	
	\begin{proof}
		1. Note that $|S_{x\langle y \rangle}|$ is even, so we may write
		    \[
		      S \, = \, xy^{\beta_1} \bd \ldots \bd xy^{\beta_{2k}} \bd y^{\beta_{2k+1}} \bd \ldots \bd y^{\beta_{n_2}} \,.
		    \]
		    Note that $\pi(S)=\{\sigma(S)\}$ is a singleton.
		    Since reordering the product must yield the same value, we have that $xy^{\beta_u}  xy^{\beta_v} = xy^{\beta_v}  xy^{\beta_u}$ if and only if $\beta_u \equiv \beta_v \pmod {n_1}$, and $y^{\beta_u}  xy^{\beta_v} = xy^{\beta_v}  y^{\beta_u}$ if and only if $\beta_u \equiv 0 \pmod {n_1}$. It follows that $\sigma(S) \in \langle y^{n_1} \rangle \cap \langle y^{n_2} \rangle = \{1_G \}$, and hence $\pi(S) = \{1_G \}$.

		2. In this case, $|S_{x\langle y \rangle}|$ is odd, and the assertion follows by a similar argument as above.
	\end{proof}

	\bigskip
	\section{A tool from Additive Theory} \label{sec:additive}

	Let us consider a metacyclic group of order $6n_2$, which is the exceptional case in \cite{AMR},
	\[
	  C_{3n_2} \rtimes_s C_2 \, = \, \langle x,y \colon x^2 = y^{3n_2} = 1_G, yx = xy^s \rangle  \, = \, \langle y^3 \rangle \times \langle x,y^{n_2} \rangle \cong \, C_{n_2} \times D_6 \,,
	\]
	where $\gcd(6,n_2) = 1$,  $s\equiv -1\pmod 3$, and $s\equiv 1 \pmod {n_2}$.
	This potential exceptional case is handled by using the DeVos-Goddyn-Mohar Theorem (see Lemma \ref{le1}), which is a generalization of the Cauchy-Davenport Theorem and Kneser Theorem (see \cite[Theorem 6.1 and Theorem 6.2]{Gry}).
	Indeed, the DeVos-Goddyn-Mohar Theorem is essential for our purposes, since it either forces the existence of a unique coset $y^{3t}H$ containing many terms of a sequence $S$ over $\langle y \rangle$, where $H$ denotes the stabilizer of the set of $(n_2-1)$-subproducts of the projection of $S$ onto $\langle y^3 \rangle$, or yields that this set of subproducts coincides with $\langle y^3 \rangle$, as will be shown in the proof of Proposition \ref{pro1}, ultimately leading to a contradiction.

	\smallskip
	\begin{lemma}[DeVos-Goddyn-Mohar {\cite[Corollary 1.5]{DGM} or \cite[Theorem 13.1]{Gry}}] \label{le1}~
		Let $G$ be a multiplicatively written finite abelian group, $S\in \mathcal F(G)$ be a sequence, $n \in [1,|S|]$, and $H = \HH \big( \Pi_n (S) \big)$. Then
		\[
		|\Pi_n (S)| \,\, \ge \,\, \min\left\{ |G| \, , \, \left( \sum_{g \in G/H} \min \{n, \mathsf v_g \big( \phi_H (S) \big) \} - n + 1 \right)|H| \right\} \,.
		\]
	\end{lemma}

	\smallskip
	The following is the main result of this section.

\smallskip
	\begin{proposition} \label{pro1}~
		Let $G = \langle y\rangle = \langle y^3\rangle \times \langle y^{n_2}\rangle$ be a cyclic group of order $3n_2$ with $\gcd(6,n_2)=1$, $G_1=\langle y^3\rangle$, $G_2= \langle y^{n_2}\rangle$, $\varphi_i \colon G \to G_i$ be projections for $i \in  [1,2]$, and $S$ be a $7n_2$-product-$G_2$ sequence over $G$.
		If, for any decomposition
		\[
		  S \, = \, T_1 \bd \ldots \bd T_7 \,\, \mbox{ with $T_i$ $n_2$-product-$G_2$ subsequences, }
		\]
		the sequence $\sigma(T_1) \bd \ldots \bd \sigma(T_7)$ has no 6-product-one subsequence, then $\Pi_{n_2-1} \big( \varphi_1(S) \big) = G_1$.
	\end{proposition}

	\begin{proof}
        If $n_2 = 1$, then $G = \langle y \rangle \cong C_3$, $G_1 = \langle y^{3} \rangle = \{ 1_G \}$, and $\Pi_{0} (S) = \{ 1_G \} = G_1$. Thus we may suppose that $n_2 \ge 2$ with $\gcd(6, n_2) = 1$, so  $G_1$ is non-trivial.
		Let $H = \HH \big( \Pi_{n_2-1} \big( \varphi_1(S) \big) \big)$, which is a subgroup of $G_1=\langle y^3\rangle$. Then $H=\langle y^{3r}\rangle$ for some $r \in [1, n_2]$ with $r \mid n_2$ and $m := |H| = \frac{n_2}{r}$.

        If $r=1$, then we are done. Suppose $r\ge 2$. Since $\gcd(6, n_2) = 1$ and $r \mid n_2$, we may assume that $r \ge 5$.

		If $\mathsf v_{g}(\phi_H(\varphi_1(S)))\le n_2-1$  for all $g\in G_1/H$,  then Lemma \ref{le1} implies that
        \begin{align*}
			n_2 \, = \, |G_1| \, &\ge \, \big|\Pi_{n_2-1} \big( \varphi_1 (S) \big)\big| \\
			&\ge \,  \min\left\{n_2 \, , \, (|S|-(n_2-1)+1)|H|\right\} \\
			&\ge \, \min\left\{n_2 \, , \, 6n_2 + 2\right\} \, = \, n_2 \,,
		\end{align*}
		whence $\Pi_{n_2-1} \big( \varphi_1 (S) \big) = G_1$.
		
		If there exist two distinct elements $g_1, g_2\in G_1/H$ such that $\mathsf v_{g_i} \big( \phi_H \big(\varphi_1(S) \big) \big) \ge n_2$ for $i \in \{1,2\}$, then Lemma \ref{le1} implies that
            \begin{align*}
				n_2 \, = \, |G_1| \, &\ge \, \big| \Pi_{n_2-1} \big( \varphi_1(S) \big) \big| \\
				&\ge \, \min\left\{ n_2 \, , \, \big( 2(n_2 - 1) -(n_2-1)+1 \big)|H| \right\} \\
				&= \, \min\left\{ n_2 \, , \, n_2 |H| \right\} \, = \, n_2 \,,
			\end{align*}
		whence $\Pi_{n_2-1} \big( \varphi_1 (S) \big) = G_1$.

        Therefore, we may assume that there exists exactly one $g\in G_1/H$ such that $\mathsf v_{g}(\phi_H(\varphi_1(S)))\ge n_2$, and it follows from Lemma \ref{le1} that
		\begin{align*}
			n_2 \, = \, |G_1| \, & \ge \, \big| \Pi_{n_2-1} \big( \varphi_1 (S) \big) \big| \\
			& \ge \, \min\left\{ n_2 \, , \, \big( |S|-\mathsf v_{g}\big( \phi_H \big( \varphi_1(S) \big) \big)+ (n_2-1) - (n_2-1)+1 \big)|H| \right\} \\
			& \ge \, \min\left\{ n_2 \, , \, \big( 7n_2 + 1 -\mathsf v_{g} \big( \phi_H \big( \varphi_1 (S) \big) \big) \big) \frac{n_2}{r} \right\} \,.
		\end{align*}
        If $\big( 7n_2 + 1 -\mathsf v_{g} \big( \phi_H \big( \varphi_1 (S) \big) \big) \big) \frac{n_2}{r} \ge n_2$, then $\Pi_{n_2-1} \big( \varphi_1 (S) \big) = G_1$, and hence the assertion follows. Thus we may assume that $n_2 \ge \big( 7n_2 + 1 -\mathsf v_{g} \big( \phi_H \big( \varphi_1 (S) \big) \big) \big) \frac{n_2}{r}$, whence $\mathsf v_{g} \big( \phi_H \big( \varphi_1 (S) \big) \big) \ge 7n_2 + 1 - r$, which implies that there exists $t \in [0, r-1]$ such that $\varphi_1 (S)$ has at least $7n_2 + 1 - r$ terms lying in the coset $y^{3t}H$.
		Now, we set
		\[
		  S \, = \, S_1 \bd S_2 \,,
		\]
		where $\varphi_1 (S_1)$ are those terms that belong to the coset $y^{3t}H$ and $\varphi_1 (S_2)$ are those terms that do not belong to the coset $y^{3t}H$. Let $\psi\colon G_1\rightarrow H$ be a map such that $\varphi_1(f)=y^{3t}\psi(f)$ for every $f \in \supp(S_1)$. Then $\psi \big( \varphi_1 (S_1) \big)$ is a sequence over $H$ of length at least $7n_2 + 1 - r \ge 2m-1$.
        By applying $\E(C_m) = 2m-1$ (Lemma \ref{inverseCn}.1), we can repeatedly choose subsequences $V_1, \ldots, V_{\nu}$ of $S_1$ such that $\psi \big( \varphi_1 (V_i) \big)$ is $m$-product-one for each $i \in [1,\nu]$, where
        \[
          \nu \, = \, \left\lceil \frac{7n_2+1-r-(2m-2)}{m} \right\rceil \, = \, 7r - 2 + \left\lceil \frac{3-r}{m} \right\rceil \,.
        \]
        Since $m \ge 1$ and $r \ge 5$, it follows that $\nu - (6r + 1) = (r-3) + \lceil \frac{3-r}{m} \rceil \ge \frac{r-3}{m} + \lceil \frac{3-r}{m} \rceil \ge 0$, whence
        \[
          \nu \, \ge \, 6r + 1\,.
        \]
		Since $\sigma \big( \varphi_1 (V_i) \big) = y^{3tm}$ for every $i \in [1, \nu]$,  we have that $\prod^{\bullet}_{i\in I}V_i$ is an $n_2$-product-$G_2$ sequence  for every $I \subset [1, \nu]$ with $|I| = r$.
		Let $[1,\nu] = I_1 \cup \cdots \cup I_7$ be a disjoint decomposition with $|I_1| = \cdots = |I_6| = r$ and $|I_7| = \nu - 6r \ge 1$.
        Now we set
        \[
          T_i \, = \, \small{\prod}^{\bullet}_{j \in I_i} V_j \,\, \mbox{ for every } \,\, i \in [1,6] \,, \quad \mbox{ and } \quad T_7 \, = \, (T_1 \bd \ldots \bd T_6)^{[-1]} \bd S \,.
        \]
        Since $S$ is a $7n_2$-product-$G_2$ sequence, we obtain that $T_1, \ldots, T_7$ are all $n_2$-product-$G_2$ sequences, and that $\prod^{\bullet}_{j \in I_7} V_j$ divides $T_7$. In view of our assumption that $\sigma(T_1) \bd \ldots \bd \sigma(T_7)$ has no $6$-product-one subsequence,
		it follows from Lemma~\ref{inverseCn}.3 that
		\begin{equation}\label{str}
	    \sigma(T_1) \bd \ldots \bd \sigma(T_7) \, = \, (y^{t_1n_2})^{[5]} \bd (y^{t_2n_2})^{[2]} \,\, \text{ for some distinct $t_1,t_2 \in [0,2]$}.
		\end{equation}
        After renumbering if necessary, we may assume that
        \[
          \sigma(T_1) \, = \, \cdots \, = \, \sigma(T_5) \, = \, y^{t_1n_2}\,.
        \]	
		Let $\lambda, \mu \in [1,7]$ be distinct, and let $k_{\lambda} \in I_{\lambda}, k_{\mu}\in I_{\mu}$.
		Then, by swapping two subsequences $V_{k_{\lambda}}$ and $V_{k_{\mu}}$, we obtain a new decomposition
		\[
		S \, = \, T'_1 \bd \ldots \bd T'_7 \,,
		\]
		where
		\begin{equation}\label{new-decomposition}
		T'_{\lambda} \, = \, V_{k_{\mu}} \bd T_{\lambda} \bd V_{k_{\lambda}}^{[-1]} \,, \quad T'_{\mu} \, = \, V_{k_{\lambda}} \bd T_{\mu} \bd V_{k_{\mu}}^{[-1]}, \quad \mbox{ and } \quad T'_j \, = \, T_j \,\, \mbox{ for all } \,\, j \in [1,7] \setminus \{ \lambda, \mu \} \,.
		\end{equation}
		Since each sequence $T'_j$ is again an $n_2$-product-$G_2$ sequence, it follows from Lemma~\ref{inverseCn}.3 and our assumption that
        \[
          {\small\prod}^{\bullet}_{j \in [1,7]} \sigma(T'_j) \, = \, (y^{t_3n_2})^{[5]} \bd (y^{t_4 n_2})^{[2]}
        \]
        for some distinct $t_3, t_4 \in [0,2]$. It follows from \eqref{new-decomposition} that $t_3 = t_1$, whence
		\begin{equation}\label{new-str}
		{\small\prod}^{\bullet}_{j \in [1,7]} \sigma(T'_j) \, = \, (y^{t_1n_2})^{[5]} \bd (y^{t_4 n_2})^{[2]} \,\, \text{ for some  $t_4 \in [0,2]\setminus \{t_1\}$} \,.
		\end{equation}
		We proceed by the following claim.
		
        \smallskip
		\begin{claim}\label{claim}
			We have
			\begin{enumerate}
				\item If $\lambda,\mu\in [1,5]$, then $\sigma(V_{k_{\lambda}})=\sigma(V_{k_{\mu}})$.
				
                \smallskip
				\item If $\{\lambda,\mu\}= \{6,7\}$, then $\sigma(V_{k_{\lambda}})=\sigma(V_{k_{\mu}})$.
				
                \smallskip
				\item If $\lambda,\mu\in [1,7]$, then $\sigma(V_{k_{\lambda}})=\sigma(V_{k_{\mu}})$.
			\end{enumerate}
		\end{claim}

		\smallskip
		If \ref{claim}.3 holds, then $\sigma(V_i)=\sigma(V_j)$ for any distinct $i,j\in [1,\nu]=I_1\cup\cdots\cup I_7$, which implies that $\sigma (T_1) = \cdots = \sigma (T_6)$, a contradiction to \eqref{str}. Thus, to finish the proof, it suffices to show \ref{claim}.

        \smallskip
		\begin{proof}[Proof of \ref{claim}]
			1. Suppose $\lambda,\mu\in [1,5]$. Then \eqref{new-decomposition} and \eqref{new-str} imply that $\sigma(T_{\lambda}')=\sigma(T_{\mu}')=\sigma(T_{\lambda})=\sigma(T_{\mu})=y^{t_1n_2}$, which implies $\sigma(V_{k_{\lambda}})=\sigma(V_{k_{\mu}})$ as desired.

			2. Suppose $\{\lambda,\mu\}= \{6,7\}$. Then \eqref{new-decomposition} and \eqref{new-str} imply that
			$\sigma(T_{\lambda}')=\sigma(T_{\mu}')=y^{t_4n_2}$ and $\sigma(T_{\lambda})=\sigma(T_{\mu})=y^{t_2n_2}$. It follows from $\sigma(T_{\lambda}')\sigma(T_{\mu})=\sigma(T_{\lambda})\sigma(T_{\mu}')$ that
		\[
		  \sigma(V_{k_{\lambda}})^{2} \, = \, \sigma(V_{k_{\mu}})^{2} \,, \,\, \text{ and hence } \,\, \sigma \big( \varphi_2(V_{k_{\lambda}}) \big)^{2} \, = \, \sigma \big( \varphi_2(V_{k_{\mu}}) \big)^{2} \, \in\,  G_2 \, \cong \, C_3 \,.
		\]
		Thus $\sigma \big( \varphi_2(V_{k_{\lambda}}) \big) = \sigma \big( \varphi_2(V_{k_{\mu}}) \big)$, together with $\sigma \big( \varphi_1(V_{k_{\lambda}}) \big) = \sigma \big( \varphi_1(V_{k_{\mu}}) \big) = y^{3tm}$, implies that $\sigma(V_{k_{\lambda}})=\sigma(V_{k_{\mu}})$ as desired.

		3. By Items 1 and 2, we may assume that $\lambda\in [1,5]$ and $\mu\in [6,7]$. By Item 1 again, we have $\sigma(V_i)=\sigma(V_j)$ for any $i,j\in I_1\cup\cdots\cup I_5$. Now applying Item 1 to  the new decomposition $S \, = \, T'_1 \bd \ldots \bd T'_7 $ instead of $S \, = \, T_1 \bd \ldots \bd T_7 $, we infer that $\sigma(V_i)=\sigma(V_j)$ for any $i,j\in \{k_{\mu}\}\cup I_1\cup\cdots\cup I_5\setminus \{k_{\lambda}\}$, whence $\sigma(V_{k_{\lambda}})=\sigma(V_{k_{\mu}})$ as desired.
		\qedhere[\ref{claim}]	
		\end{proof}
    \end{proof}

	\bigskip
	\section{Proof of Theorem~\ref{thm:main}} \label{sec:proofmainresult}

	We first show the following proposition, which plays a key role in both parts of the proof of  Theorem \ref{thm:main}.

	\smallskip
	\begin{proposition} \label{prop:main}~
		Let $G = \langle x, y \colon x^2 = y^{3n_2} = 1_G, yx = xy^s \rangle \cong C_{n_2} \times D_6$, where $n_2\ge 5$ with $\gcd(6, n_2) = 1$,   $s\equiv -1\pmod 3$, and $s\equiv 1 \pmod {n_2}$.
		If $S$ is a sequence over $G$ of length $9n_2-1$ with $|S_{x\langle y\rangle}|\ge 2$, then $S$ has a product-one subsequence of length $6n_2$.
	\end{proposition}

	\begin{proof}Note that $\langle y^{3}\rangle\cong C_{n_2}$ and $\langle x, y^{n_2} \rangle\cong D_6$.
		Let $\varphi\colon G \to \langle y^{3}\rangle$ be the projection. Then $\varphi (S)$ is a sequence over $\langle y^{3}\rangle$ of length $9n_2 - 1$. It follows from $\E(C_{n_2})=2n_2-1$ (Lemma \ref{inverseCn}.1) that we can repeatedly choose $n_2$-product-$\langle x, y^{n_2} \rangle$ subsequences $T_1,\ldots, T_8$ of $S$, that is, every $\varphi(T_i)$ is an $n_2$-product-one sequence.
		Among all the possible  choices of $T_1,\ldots, T_8$, we may assume that
		\begin{equation}\label{maxi}
			\big| \{i\in [1,8]\colon |(T_i)_{x\langle y\rangle}|>0\} \big| \, \text{ is maximal}.
		\end{equation}
        Then
        \[
          T^* \, := \, \sigma(T_1)\,\bd\ldots\bd\, \sigma(T_8)
        \]
        is a sequence over $\langle x, y^{n_2} \rangle$, where $\sigma(T_i)\in \pi(T_i)$ is chosen arbitrarily.
		If $T^*$ has a $6$-product-one subsequence, then we are done. Thus, we can assume that
		\begin{equation}\label{6-free}
		\text{$T^*$ is $6$-product-one free for any choice of $\sigma (T_i)\in \pi (T_i)$.}
		\end{equation}

        We distinguish two cases.

		\smallskip
		\noindent
		{\bf Case 1 :} $|(T^{*})_{x \langle y \rangle}| \ge 2$.
		\smallskip

        Then Lemma~\ref{lem:D6}.2 implies that $T^* = 1_{G}^{[5]} \bd x \bd xy^{n_2} \bd xy^{2n_2}$. After renumbering if necessary, we may assume that $\sigma(T_i) = 1_G$ for all $i\in [1,5]$.
        Then $(T_i)_{x\langle y\rangle}$ is nontrivial for $i \in [6,8]$.
        Note that $\pi (T_i) \subseteq \{ x, xy^{n_2}, xy^{2n_2} \}$ for each $i \in [6,8]$. If $|\pi (T_i)| \ge 2$ for some $i \in [6,8]$, then $\pi (T_i)\cap \big\{ \sigma (T_j) \colon j\in [6,8]\setminus \{i\} \big\} \neq \emptyset$, and by choosing $\sigma^*\in \pi(T_i)\cap \big\{ \sigma (T_j)\colon j\in [6,8]\setminus \{i\} \big\}$,  we obtain  $\sigma^*\bd T^*\bd \sigma(T_i)^{[-1]}$ has a 6-product-one sequence $1_G^{[4]}\bd (\sigma^*)^{[2]}$, a contradiction to \eqref{6-free}.
        Thus
        \begin{equation}\label{sig}
        \text{ $| \pi (T_i)| \, = \, 1$ \,\, for all \, $i \in [6,8]$. }
        \end{equation}
         In view of  Lemma \ref{lem:claima}.2 and $\sigma(T_6)\bd \sigma(T_7)\bd \sigma(T_8)=x \bd xy^{n_2} \bd xy^{2n_2}$, we may assume that there exist positive integers $u,v,w$ such that
        \begin{equation} \label{eq:mod3}
		\begin{aligned}
			T_6 \, & = \, xy^{\beta_1} \bd \ldots \bd xy^{\beta_{2u-1}} \bd y^{\beta_{2u}} \bd \ldots \bd y^{\beta_{n_2}} \,, \\
			       & \quad \, \text{ where } \, \beta_1 \, \equiv \, \cdots \, \equiv \, \beta_{2u-1}  \, \pmod 3 \,\, \text{ and } \,\, \beta_{2u} \, \equiv \, \cdots \, \equiv \, \beta_{n_2} \, \equiv \, 0 \pmod 3 \,, \\
			T_7 \, & = \, xy^{\gamma_1} \bd \ldots \bd xy^{\gamma_{2v-1}} \bd y^{\gamma_{2v}} \bd \ldots \bd y^{\gamma_{n_2}} \,, \\
			       & \quad \, \text{ where } \, \gamma_1 \, \equiv \, \cdots \, \equiv \, \gamma_{2v-1} \,  \pmod 3 \,\, \text{ and } \,\, \gamma_{2v} \, \equiv \, \cdots \, \equiv \, \gamma_{n_2} \, \equiv \, 0 \pmod 3 \,, \\
			T_8 \, & = \, xy^{\delta_1} \bd \ldots \bd xy^{\delta_{2w-1}} \bd y^{\delta_{2w}} \bd \ldots \bd y^{\delta_{n_2}} \,, \\
			       & \quad \, \text{ where } \, \delta_1 \, \equiv \, \cdots \, \equiv \, \delta_{2w-1} \,  \pmod 3 \,\, \text{ and } \,\, \delta_{2w} \, \equiv \, \cdots \, \equiv \, \delta_{n_2} \, \equiv \,  0 \pmod 3 \,.
		\end{aligned}
        \end{equation}
    Furthermore, we have $\sigma(T_6)=xy^{\beta_1+3t_6}=xy^{\epsilon_6n_2}$, $\sigma(T_7)=xy^{\gamma_1+3t_7}=xy^{\epsilon_7n_2}$, and $\sigma(T_8)=xy^{\delta_1+3t_8}=xy^{\epsilon_8n_2}$, where $t_6,t_7, t_8\in \Z$ and $\{\epsilon_6, \epsilon_7,\epsilon_8\}= [0,2]$. Since $n_2\not\equiv 0 \pmod {3}$, after renumbering if necessary, we may assume that
    \begin{equation}\label{equiv}
      \beta_1 \, \equiv \, 0 \, \pmod 3 \,, \quad \gamma_1 \, \equiv \, 1 \, \pmod 3 \,, \quad \mbox{ and } \quad \delta_1 \, \equiv \, 2 \, \pmod 3 \,.
    \end{equation}
		Since $(xy^{\beta_1} \bd xy^{\gamma_1} \bd xy^{\delta_1})^{[-1]} \bd T_6 \bd T_7 \bd T_8$ has length $3n_2-3\ge 2n_2-1$, there exists an $n_2$-product-$\langle x,y^{n_2} \rangle$ subsequence $T_6'$.
        Since $\varphi \big( (T_6')^{[-1]} \bd T_6 \bd T_7 \bd T_8 \big)$ is a $2n_2$-product-one sequence over $\langle y^3\rangle$,  it is a product of two $n_2$-product-one subsequences, say $T'_7$ and $T'_8$, i.e.,
        \[
          T_6 \bd T_7 \bd T_8 \, = \, T_6' \bd T_7' \bd T_8' \,\, \mbox{ for $n_2$-product-$\langle x, y^{n_2} \rangle$ subsequences $T_i'$.}
        \]
        We obtain a new sequence
        \[
          (T^*)' \, = \, 1^{[5]}_G \bd \sigma(T_6') \bd \sigma(T_7') \bd \sigma(T_8')
        \]
        over $\langle x, y^{n_2} \rangle$, where $\sigma(T_i')\in \pi(T_i')$ is chosen arbitrarily.
        If $(T^*)'$ has a 6-product-one subsequence, then we are done, whence we may assume that
        \begin{equation}\label{6-free*}
        \text{$(T^*)'$ has no 6-product-one subsequence for any choice of $\sigma(T_i')\in \pi(T_i')$.}
        \end{equation}

        If $\big|(T^*)'_{x \langle y \rangle} \big| \ge 2$, then it follows from Lemma \ref{lem:D6}.2 that $(T^*)' = T^*$.
        Since $xy^{\beta_1} \bd xy^{\gamma_1} \bd xy^{\delta_1} \mid T'_7 \bd T'_8$, we infer that either $T'_7$ or $T'_8$, say $T'_7$, must contain at least two terms among $xy^{\beta_1}, xy^{\gamma_1}, xy^{\delta_1}$.
        By symmetry, we may assume that $xy^{\beta_1} \bd xy^{\gamma_1} \mid T'_7$.
        Applying \eqref{sig} to the new sequence $(T^*)'$ instead of $T^*$, we obtain $| \pi (T'_7)| = 1$. By Lemma \ref{lem:claima}.2 again, we obtain $\beta_1 \equiv \gamma_1 \pmod{3}$, a contradiction to \eqref{equiv}.

        Therefore we must have that $\big| (T^*)'_{x \langle y \rangle} \big| = 1$, and hence it follows again by Lemma~\ref{lem:D6}.2 that $\sigma(T_6') \bd \sigma(T_7') \bd \sigma(T_8') = (y^{sn_2})^{[2]} \bd xy^{kn_2}$ for some $s \in [1,2]$ and $k \in [0,2]$.
        We may assume that $\sigma (T'_8) = xy^{kn_2}$. Note that $\pi (T'_i) \subseteq \{ 1_G, y^{n_2}, y^{2n_2} \}$ for $i \in [6,7]$.  If $|\pi(T'_i)| \ge  2$ for some $i \in [6,7]$,
        then letting $\sigma^*\in \pi(T_i')\setminus \{\sigma(T_i')\}$ implies that either $\sigma^*=1_G$ or $\sigma^*\sigma(T_j')=1_G$, where $j\in [6,7]\setminus \{i\}$. Both cases imply that $\sigma^*\bd (T^*)'\bd (\sigma(T_i'))^{[-1]}$ has a 6-product-one subsequence, a contradiction to \eqref{6-free*}.
        Then Lemma~\ref{lem:claima}.1 implies
        \[
          \text{ $\big| (T_6')_{x\langle y \rangle} \big| \, = \, \big| (T_7')_{x\langle y \rangle} \big| \, = \, 0$\,, \, and hence $xy^{\beta_1} \bd xy^{\gamma_1} \bd xy^{\delta_1} \mid T_8'$. }
        \]
		In view of (\ref{eq:mod3}), we have that
		\[
          \big|(T_6 \bd T_7 \bd T_8)_{\langle y^3 \rangle} \big| \, = \, \big| (T_6' \bd T_7' \bd T_8')_{\langle y^3 \rangle} \big| \, \ge \, \big| (T_6' \bd T_7')_{\langle y^3 \rangle} \big| \, = \, \big| (T_6' \bd T_7')_{\langle y \rangle} \big| \, = \, 2n_2 \,,
        \]
		which ensures by $\mathsf E (C_{n_2}) = 2n_2 - 1$ that $(T_6 \bd T_7 \bd T_8)_{\langle y^3 \rangle}$ has an $n_2$-product-$\langle x,y^{n_2} \rangle$ subsequence $T_0$.
        Thus $\sigma(T_0) \in \langle y^3 \rangle \cap \langle x,y^{n_2} \rangle = \{ 1_G \}$, whence $T_1 \bd \ldots \bd T_5 \bd T_0$ is a $6n_2$-product-one subsequence of $S$.

		\smallskip
		\noindent
		{\bf Case 2 :} $\big| (T^*)_{x \langle y \rangle} \big| = 1$.
		\smallskip

        By Lemma~\ref{lem:D6}.2, we have $T^*= (y^{t_1n_2})^{[5]}\bd (y^{t_2n_2})^{[2]} \bd (xy^{t_3n_2})$ for $t_1, t_2, t_3 \in [0,3]$ with $\gcd (t_1 - t_2, 3) = 1$.
		After renumbering if necessary, we may assume that $\sigma(T_i)=y^{t_1n_2}$ for all $i\in [1,5]$ and $\sigma(T_j)= y^{t_2n_2}$ for all $j\in [6,7]$.
        Note that $\pi (T_k) \subseteq \{ 1_G, y^{n_2}, y^{2n_2} \}$ for all $k \in [1,7]$. If there exists $k\in [1,7]$ such that $|\pi(T_k)|\ge 2$, then,  letting  $\sigma^*\in \pi(T_i)\setminus \{\sigma(T_i)\}$, it follows from Lemma~\ref{lem:D6}.2 and \eqref{6-free} that $\sigma^*\bd T^*\bd (\sigma(T_i))^{[-1]}=T^*$, whence $\sigma^*=\sigma(T_i)$, a contradiction.
        Thus
        \[
          |\pi(T_k)| \, = \, 1 \,\, \mbox{ for all } \, k \in [1,7] \,.
        \]

        \smallskip
        \noindent
        {\bf Subcase 2.1 :} $\big| (T_1\bd \ldots \bd T_7)_{x\langle y\rangle} \big| \ge 1$.
        \smallskip

		If $ \big| (T_1 \bd \ldots \bd T_5)_{x \langle y \rangle} \big| \ge 1$, say $ |(T_1)_{x\langle y\rangle}|\ge 1$, then Lemma \ref{lem:claima}.1 implies  that $\sigma(T_1) = 1_G$ and hence
		\[
		  \sigma(T_i) \, = \, 1_G \,\, \mbox{ for all } \, i\in [1,5] \,.
		\]
		Let $h$ be  a term of $(T_1)_{x\langle y\rangle}$. Then $h\,\sigma (h^{[-1]} \bd T_1)=1_G$ and hence it follows from $s\equiv -1\pmod 3$ that
		\begin{align*}
          & \, \sigma(T_6)\, h\, \sigma(T_7) \,\sigma (h^{[-1]} \bd T_1) \sigma(T_2) \sigma(T_3) \sigma(T_4) \\
          = & \, y^{t_2n_2}\,h\,y^{t_2n_2}\,\sigma (h^{[-1]} \bd T_1)\\
          = & \, h\,y^{t_2n_2s+t_2n_2}\,\sigma (h^{[-1]} \bd T_1)\\
          = & \, h\,\sigma (h^{[-1]} \bd T_1)\\
          = & \, 1_G \,,
        \end{align*}
		whence $T_1\bd \ldots\bd T_4 \bd T_6 \bd T_7$ is a $6n_2$-product-one subsequence of $S$.
        If $\big| (T_6 T_7)_{x\langle y\rangle} \big| \ge 1$, say $\big| (T_6)_{x\langle y\rangle} \big| \ge 1$, then  Lemma~\ref{lem:claima} again implies that $\sigma(T_i) = 1_G$ for all $i\in [6,7]$. Let $h$ be a term of $(T_6)_{x\langle y\rangle}$.
It follows from $s\equiv -1\pmod 3$ that
\begin{align*}
	& \, \sigma(T_1) \sigma (T_2)\, h\, \sigma(T_3) \sigma (T_4)\, \sigma (h^{[-1]} \bd T_6)\, \sigma (T_7) \\
	= & \, y^{2t_1n_2}\,h\,y^{2t_1n_2}\,\sigma (h^{[-1]} \bd T_6)\\
	= & \, h\,y^{2t_1n_2s+2t_1n_2}\,\sigma (h^{[-1]} \bd T_6)\\
	= & \, h\,\sigma (h^{[-1]} \bd T_6)\\
	= & \, 1_G \,,
\end{align*}
whence $T_1\bd \ldots\bd T_4 \bd T_6 \bd T_7$ is a $6n_2$-product-one subsequence of $S$.

        \smallskip
        \noindent
        {\bf Subcase 2.2 :} $\big| (T_1\bd \ldots \bd T_7)_{x\langle y\rangle} \big| = 0$.
        \smallskip

		Then $Y:=T_1 \bd \ldots \bd T_7$ is a sequence over $\langle y \rangle$ of length $7n_2$ and
		\[
			\big| \{i\in [1,8]\colon |(T_i)_{x\langle y\rangle}| > 0 \} \big| \, = \, 1 \,.
		\]
		If there exists a decomposition $Y=Z_1\bd \ldots \bd Z_7$ into a product of  $n_2$-product-$\langle x, y^{n_2} \rangle$ subsequences $Z_i$ such that $\sigma (Z_1) \bd \ldots \bd \sigma (Z_7)$ has a 6-product-one subsequence, then we are done.
		Thus we may suppose that $\sigma (Z_1) \bd \ldots \bd \sigma (Z_7)$ has no 6-product-one subsequence for any decomposition of $Y=Z_1\bd \ldots \bd Z_7$ into a product of  $n_2$-product-$\langle x, y^{n_2} \rangle$ subsequences $Z_i$.
		In view of Proposition \ref{pro1}, we infer that
		\begin{equation} \label{eq:pi}~
		  \Pi_{n_2 - 1} \big( \varphi (T_1 \bd \ldots \bd T_7) \big) \, = \, \langle y^{3} \rangle \,.
		\end{equation}

		Let $E = S \bd (T_1 \bd \ldots \bd T_8)^{[-1]}$.
		Suppose $\big| E_{x\langle y\rangle } \big| \ge 1$. Let $h$ be a term of  $E_{x\langle y\rangle }$. Since $\varphi(h)\in \langle y^3\rangle$,  it follows from \eqref{eq:pi}  that $T_1 \bd \ldots \bd T_7$ has a subsequence $T_0$ of length $n_2 -1$ such that $\varphi (h \bd T_0 )$ is a product-one sequence, whence $T_8^{[-1]} \bd S$ has an $n_2$-product-$\langle x,y^{n_2} \rangle$ subsequence $T_1'$ with $h \mid T_1'$.
		By applying $\mathsf E (C_{n_2}) = 2n_2 - 1$, we obtain a new decomposition of $S$
		\[
		  S \, = \, T_1'  \bd T_8 \bd T_2' \bd \ldots  \bd T_7' \bd E' \,,
		\]
		where both $T_1'$ and $T_8$ have terms from $x \langle y\rangle$, a contradiction to \eqref{maxi}.

		Thus $|E_{x\langle y\rangle }| = 0$, whence $\big| S_{x\langle y\rangle} \big| = \big| (T_8)_{x\langle y\rangle} \big| \ge 2$. Let $h$ and $h'$ be two terms with $h \bd h'\mid (T_8)_{x\langle y\rangle}$.
		In view of (\ref{eq:pi}), we infer that  $T_1 \bd \ldots \bd T_7$ has a subsequence $T_0$ of length $n_2 - 1$ such that $\varphi (h \bd T_0)$ is a product-one sequence, whence $(h')^{[-1]} \bd S$ has an $n_2$-product-$\langle x,y^{n_2} \rangle$ subsequence $T_1'$ with $h \mid T_1'$.
		By applying  $\mathsf E (C_{n_2}) = 2n_2 - 1$, we  obtain a new decomposition of $S$
		\[
		  S \, = \, T_1' \bd T_2' \bd \ldots \bd T_8' \bd E' \,,
		\]
		where  $h \mid T_1'$ and $h'\mid T_2' \bd \ldots \bd T_8' \bd E'$.
	    By \eqref{maxi}, we must have $h'\mid E'$ and go back to  the previous case that $\big| E'_{x\langle y\rangle } \big| \ge 1$.
		This completes our proof.
	\end{proof}

	\smallskip
	We are now able to prove the main theorem of this paper.

	\medskip
	\begin{proof}[Proof of Theorem~\ref{thm:main}]
		Since $n_2\neq 1$ and $\gcd(n_2,6)=1$, we have  that $n_2\ge 5$.
		
		1. It is known that $\mathsf E(G)\ge 9n_2$ (see \cite[Lemma 4]{ZhGa}). Let $S \in \F(G)$ be a sequence of length $|S| = 9n_2$.
		If $\big| S_{x\langle y \rangle} \big| \le 1$, then $S_{\langle y\rangle }$ is a sequence over $\langle y\rangle$ of length at least $9n_2-1$ and hence Lemma \ref{inverseCn}.2 ensures that $S_{\langle y \rangle}$ has a $6n_2$-product-one subsequence.
		If $\big| S_{x\langle y \rangle} \big| \ge 2$, then Proposition \ref{prop:main} ensures that $S$ has a $6n_2$-product-one subsequence.
		Therefore $\E(G) \le  9n_2$ and hence $\mathsf E(G)= 9n_2$.

		2. Let $S \in \F(G)$ be a sequence of length $|S| = 9n_2-1$.
		If $S_{\langle y \rangle} = S$, then Lemma \ref{inverseCn}.2 implies that $S$ has a $6n_2$-product-one subsequence.
		Thus we may assume that $|S_{x\langle y \rangle}| \ge 1$.
		If $\big| S_{x\langle y \rangle} \big| \ge  2$, then Proposition~\ref{prop:main} implies that $S$  have a $6n_2$-product-one subsequence. Thus
        \[
          |S_{x\langle y \rangle}| \, = \, 1 \,.
        \]
		Hence $\big| S_{\langle y \rangle} \big| = 9n_2 - 2$ and it follows from  Lemma \ref{inverseCn}.3 that there exist $t_1, t_2 \in \Z$  with $\gcd(t_1 - t_2,3n_2)=1$ such that $S_{\langle y \rangle} = (y^{t_1})^{[6n_2-1]} \bd (y^{t_2})^{[3n_2-1]}$.
		The assertion now follows.
	\end{proof}

	\bigskip
	\noindent{\it Concluding remark:}
		It is worth mentioning that the direct and inverse problems related to $\E(G)$ are already solved for the non-split metacyclic groups of the form $\langle x,y \colon x^2 = y^n, y^{2n} = 1_G, yx = xy^s \rangle$, where $s^2 \equiv 1 \pmod {2n}$ (see \cite{Rib,YZF}). The remaining case to complete the direct and inverse problems related to $\E(G)$ for all groups having a cyclic subgroup of index $2$ is when $G = \langle x, y \colon x^{2} = y^{m}, y^{n} = 1_G, yx = xy^{s} \rangle$, where $n = 2^{t}m$ with $t \ge 2$ and $m$ odd.

	\smallskip
	\section*{Acknowledgements}
	This paper was written while J.S. Oh, S. Ribas, and K. Zhao were visiting the Algebra and Number Theory Research Group at the University of Graz in July 2025. The authors are grateful for the hospitality of the group. The problem studied in this paper was proposed by S. Ribas during the seminar ``Problem Session on Additive Combinatorics'', which initiated the collaboration among the authors (see \url{https://imsc.uni-graz.at/AlgNTh/seminar.html}).
	
    J.S. Oh was supported by the National Research Foundation of Korea (NRF) grant funded by the Korea government (MSIT) (NRF-2021R1G1A1095553).
	S. Ribas was partially supported by FAPEMIG grants RED-00133-21, APQ-01712-23 and APQ-04712-25, and CNPq grant 420721/2025-8. K. Zhao  was supported  by National Science Foundation
	of China, Grant \#12301425. Q. Zhong was supported by the Austrian Science Fund FWF, Project P36852-N (Doi: 10.55776/P36852).

\end{document}